\documentclass{amsart} 
\usepackage{amscd,amssymb,latexsym,subfigure,verbatim,epsfig,amsmath,supertabular}
\usepackage[graph,frame,poly,arc]{xy}
\usepackage{calc}
\usepackage{amsmath}
\usepackage{graphicx}

\begin{document}

\newcommand{\mmbox}[1]{\mbox{${#1}$}}
\newcommand{\proj}[1]{\mmbox{{\mathbb P}^{#1}}}
\newcommand{\affine}[1]{\mmbox{{\mathbb A}^{#1}}}
\newcommand{\Ann}[1]{\mmbox{{\rm Ann}({#1})}}
\newcommand{\caps}[3]{\mmbox{{#1}_{#2} \cap \ldots \cap {#1}_{#3}}}
\newcommand{\N}{{\mathbb N}}
\newcommand{\Z}{{\mathbb Z}}
\newcommand{\R}{{\mathbb R}}
\newcommand{\K}{{\mathbb K}}
\newcommand{\p}{{\mathbb P}}
\newcommand{\A}{{\mathcal A}}
\newcommand{\CC}{{\mathcal C}}
\newcommand{\C}{{\mathbb C}}
\newcommand{\CR}{C^r(\hat P)}

\newcommand{\Tor}{\mathop{\rm Tor}\nolimits}
\newcommand{\ann}{\mathop{\rm ann}\nolimits}
\newcommand{\Ext}{\mathop{\rm Ext}\nolimits}
\newcommand{\Hom}{\mathop{\rm Hom}\nolimits}
\newcommand{\im}{\mathop{\rm Im}\nolimits}
\newcommand{\rank}{\mathop{\rm rank}\nolimits}
\newcommand{\supp}{\mathop{\rm supp}\nolimits}
\newcommand{\arrow}[1]{\stackrel{#1}{\longrightarrow}}
\newcommand{\CB}{Cayley-Bacharach}
\newcommand{\coker}{\mathop{\rm coker}\nolimits}
\sloppy
\newtheorem{defn0}{Definition}[section]
\newtheorem{prop0}[defn0]{Proposition}
\newtheorem{conj0}[defn0]{Conjecture}
\newtheorem{thm0}[defn0]{Theorem}
\newtheorem{lem0}[defn0]{Lemma}
\newtheorem{corollary0}[defn0]{Corollary}
\newtheorem{example0}[defn0]{Example}

\newenvironment{defn}{\begin{defn0}}{\end{defn0}}
\newenvironment{prop}{\begin{prop0}}{\end{prop0}}
\newenvironment{conj}{\begin{conj0}}{\end{conj0}}
\newenvironment{thm}{\begin{thm0}}{\end{thm0}}
\newenvironment{lem}{\begin{lem0}}{\end{lem0}}
\newenvironment{cor}{\begin{corollary0}}{\end{corollary0}}
\newenvironment{exm}{\begin{example0}\rm}{\end{example0}}

\newcommand{\defref}[1]{Definition~\ref{#1}}
\newcommand{\propref}[1]{Proposition~\ref{#1}}
\newcommand{\thmref}[1]{Theorem~\ref{#1}}
\newcommand{\lemref}[1]{Lemma~\ref{#1}}
\newcommand{\corref}[1]{Corollary~\ref{#1}}
\newcommand{\exref}[1]{Example~\ref{#1}}
\newcommand{\secref}[1]{Section~\ref{#1}}
\newcommand{\poina}{\pi({\mathcal A}, t)}
\newcommand{\poinc}{\pi({\mathcal C}, t)}
\newcommand{\std}{Gr\"{o}bner}
\newcommand{\jq}{J_{Q}}

\title {Piecewise polynomials on polyhedral complexes}

\author{Terry Mcdonald}
\address{McDonald: Mathematics Department \\ Midwestern State University \\
  Wichita Falls \\ TX 76308\\ USA}
\email{ terry.mcdonald@mwsu.edu}

\author{Hal Schenck}
\thanks{Schenck supported by  NSF 0707667, NSA H98230-07-1-0052}\address{Schenck: Mathematics Department \\ University of Illinois Urbana-Champaign\\
  Urbana \\ IL 61801\\ USA}
\email{schenck@math.uiuc.edu}

\subjclass[2000]{Primary 41A15, Secondary 13D40, 52C99} \keywords{polyhedral
  spline, dimension formula, Hilbert polynomial.}

\begin{abstract}
\noindent For a $d$-dimensional polyhedral complex $P$, the dimension
of the space of piecewise polynomial functions (splines) on $P$ of 
smoothness $r$ and degree $k$ is given, for $k$ sufficiently large,
by a polynomial $f(P,r,k)$ of degree $d$. When $d=2$ and $P$ is
simplicial, in \cite{as} Alfeld and Schumaker give a formula for 
all three coefficients of $f$. However, in the polyhedral case,
no formula is known. Using localization techniques and specialized
dual graphs associated to codimension--$2$ linear spaces, we obtain
the first three coefficients of $f(P,r,k)$, giving a complete
answer when $d=2$.
\end{abstract}
\maketitle


\section{Introduction}\label{sec:intro}
In \cite{b}, Billera used methods of homological
and commutative algebra  to solve a conjecture of Strang 
on the dimension of the space of $C^1$ splines on a generic planar simplicial
complex. The algebraic approach to the study of splines was further 
developed in work of Billera and Rose \cite{br}, \cite{br1}, and Schenck and 
Stillman \cite{ss1}, \cite{ss2}.
For a (not necessarily generic) triangulation $\Delta$ of a simply connected 
polygonal domain having $f_1^0$ interior edges and $f_0^0$ interior vertices, 
Alfeld and Schumaker showed that for $k\ge 3r+1$, the dimension of the 
vector space $C^r_k(\Delta)$ of splines on $\Delta$ having 
smoothness $r$ and degree at most $k$ is given by:
\begin{thm}[\cite{as1}]
\vskip .1in
\[
\dim C^r_k(\Delta) = {k+2 \choose 2} + {k-r+1 \choose 2}f^0_1 - \left( {k+2 \choose 2} - {r+2 \choose
    2}\right)f^0_0 + \sigma,
\]
\vskip .1in
\noindent where $\sigma = \sum\sigma_i$, $\sigma_i = \sum_j \max\{
(r+1+j(1-n(v_i))), 0 \}$, 
and $n(v_i)$ is the number of distinct slopes at an interior
vertex $v_i$. 
\end{thm}
In higher dimensions, a spectral sequence argument \cite{s1} 
shows that the previous theorem continues to 
hold, in the sense that for $k \gg 0$, the dimension of 
$C^r_k(\Delta)$ is given by a polynomial in $k$ (the {\em Hilbert
polynomial}), and the coefficients of the three largest 
terms of this polynomial are determined by a higher dimensional
analog of the formula given by Alfeld and Schumaker in the planar case.

In this note, we study splines on a {\em polyhedral} complex $P$. 
In general, polyhedral complexes do not lend
themselves to the full range of techniques available in the simplicial
case, and so have been studied less. In \cite{schu1}, Schumaker 
obtains upper and lower bounds in the planar case, and in \cite{br}, 
Billera and Rose obtained the first two coefficients of the Hilbert
polynomial. It is possible to consider all the $C^r_k(P)$ at
once by passing to a module over a polynomial ring, 
and in \cite{y} Yuzvinsky uses sheaves on 
posets to obtain results on the freeness of this module. In \cite{r1},
\cite{r2}, Rose studies cycles on the dual graph of $P$. Our key 
technical innovation is a refined version of the dual graph, depending
on a choice of codimension--$2$ linear space; used in conjunction with
localization techniques.

The main result of this paper is a formula for the third 
coefficient of the Hilbert polynomial, in the case of a polyhedral
complex. In particular, for a planar polyhedral complex, our
result gives a formula for the dimension of $C^r_k(P)$ for
all $k\gg 0$.
\begin{exm}\label{exm:first}
Let $P$ be the polygonal complex depicted below:
\begin{figure}[h]
\begin{center}
\epsfig{file=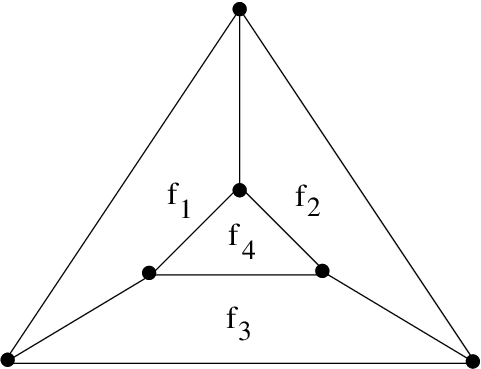,height=1.3in,width=1.5in}
\end{center}
\end{figure}

\noindent The following table gives the dimensions for $C^r_k(P)$ for small values
of $r$, assuming $k \gg 0$. For a planar $P$, Corollary
\ref{cor:planarHP} gives a formula for $f(P,r,k)$ depending
on the combinatorics and geometry of $P$, ``explaining'' the 
values below.
\vskip .1in
\begin{center}
\begin{supertabular}{|c|c|} 
\hline $r$ & $\dim_{\mathbb{R}} C^r_k(P)$ \\ 
\hline $0$ & $2k^2+2$ \\ 
\hline $1$ & $2k^2-6k+10$ \\ 
\hline $2$ & $2k^2-12k+32$ \\ 
\hline $3$ & $2k^2-18k+64$ \\ 
\hline $4$ & $2k^2-24k+110$ \\ 
\hline
\end{supertabular}
\end{center}
\vskip .1in
\end{exm}
Throughout the paper, our basic references are de Boor \cite{dB} (for
splines) and  Eisenbud \cite{E} (for commutative algebra).
\section{Algebraic preliminaries}
\begin{defn}\label{defn:dualG}
The dual graph $G_P$ of a $d$--dimensional polyhedral complex $P$ is a graph
whose vertices correspond to the $d$--faces of $P$, with an edge 
connecting two vertices iff the corresponding two $d$--faces meet 
in a $(d-1)$--face.
\end{defn}
\begin{defn}\label{defn:star}
The star of a face $\sigma$ of $P$ is the set of all faces $\tau \in
P$ such that there exists a face $\upsilon \in P$ containing both
$\tau$ and $\sigma$. 
\end{defn}
Put another way, the star of $\sigma$ is the smallest 
subcomplex of $P$ containing all faces which contain $\sigma$.
Let $P$ be a $d$-dimensional polyhedral complex embedded in $\mathbb{R}^d$,
such that $P$ is {\em hereditary}; this means for every nonempty face
$\sigma$ of $P$, the dual graph of the star of $\sigma$ is connected.
Basically, $P$ should be visualized as a polyhedral subdivision of a 
manifold with boundary. A $C^r$--{\it spline} on $P$ is a piecewise 
polynomial function (a polynomial is assigned to each $d$-dimensional
cell of $P$), such that two polynomials supported on $d$-faces 
which share a common $(d-1)$--face $\tau$ meet with order of
smoothness $r$ along that face. 
We use $P^0_i$ to denote the set of interior $i$ faces of $P$ 
(all $d$--dimensional faces are considered interior), and let
$f^0_i= |P^0_i| $. The set of splines of degree at most $k$ (i.e. each
individual polynomial is of degree at most $k$) is a vector space, which
we will denote $C^r_k(P)$.

\subsection{Splines and syzygies}
The first important fact is that 
the smoothness condition is local: for two $d$--cells 
$\sigma_1$ and $\sigma_2$ sharing a common 
$(d-1)$--face $\tau$, let $l_\tau$ be
a nonzero linear form vanishing on $\tau$. Billera and Rose
(\cite{br}, Corollary 1.3) show that a pair of polynomials
$f_i$ supported on $\sigma_i, i=1,2$ meet with order $r$ 
smoothness along $\tau$ iff
$$l_\tau^{r+1} | f_1-f_2. $$
In \cite{schu}, Schumaker gave a formula for the dimension of
$C^r_k(P)$ in the planar case, when $P$ is the
star of a vertex, as depicted in Example \ref{exm:syz} (note
that such a $P$ will always be simplicial).
Schumaker observed that a spline is a syzygy on
the ideal generated by the $r+1^{st}$ powers of linear forms vanishing
on the interior edges. To see this, consider the example below:

\begin{exm}\label{exm:syz} A planar $P$ which is the star of a single interior vertex $v_
0$ at the origin.
\begin{center}
\epsfig{figure=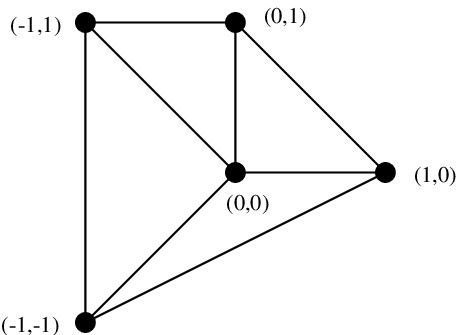,height=1.5in,width=2.0in}
\end{center}

Beginning
with the simplex in the first quadrant and moving clockwise, label
the polynomials on the triangles $f_1,\ldots, f_4$. To obtain a
global $C^r$ function, we require
\begin{center}
$\begin{array}{ccc}
a_1 y^{r+1}&=&f_1-f_2\\
a_2 (x-y)^{r+1}&=&f_2-f_3\\
a_3 (x+y)^{r+1}&=&f_3-f_4\\
a_4 x^{r+1}&=&f_4-f_1
\end{array}$
\end{center}
Summing each side above yields the equation
$\sum_{i=1}^4a_i l_i^{r+1} = 0$, a syzygy on the ideal
$$I = \langle y^{r+1}, (x-y)^{r+1}, (x+y)^{r+1}, x^{r+1} \rangle.$$
It is easy to see that the process can be reversed. The upshot is
that if $P$ is the star of a vertex, then the spline module consists
of $\mathbb{R}[x,y] \oplus syz(I)$; the summand $\mathbb{R}[x,y]$ corresponds
to the ``trivial'' splines $f_1=f_2=f_3=f_4$. Hence, splines are
intimately connected to commutative algebra.
\end{exm}
\subsection{Graded modules, Hilbert polynomial and series}
As first observed by Billera and Rose, one way to study the 
dimension of the vector space $C^r_k(P)$ is to embed $P$ in the hyperplane 
$x_{d+1} = 1 \subseteq \mathbb{R}^{d+1}$, and form the cone $\hat P$
over $P$, with vertex at the origin. If $C^r_k(\hat P)$ is the 
set of splines on $\hat P$ such that each polynomial 
is homogeneous of degree $k$, then Billera and Rose show (\cite{br},
Theorem 2.6) there 
is a vector space isomorphism between $C^r_k(\hat P)$ and $C^r_k(P)$. 
It is easy to see that the set of splines {\em of all degrees} 
\[
C^r(\hat P) = \bigcup\limits_{k \ge 0}C^r_k(\hat P)
\]
is a graded module over the polynomial ring $R$ in $d+1$ variables. 
It follows from basic commutative algebra (see \cite{sch}) that 
the dimension of all the graded pieces of
$C^r(\hat P)$ is encoded by the {\em Hilbert series}:
\[
HS(C^r(\hat P),t) = \sum_k \mbox{dim}C^r_k(\hat P)t^k = \frac{g(t)}{(1-t)^{d+1}};
\]
for some $g(t) \in \mathbb{Z}[t]$, and that for $k\gg 0$, the 
dimension of $C^r_k(P)$ is given by the {\em Hilbert
  polynomial} $HP(C^r(\hat P),k)$, which is an element of 
$\mathbb{Q}[k]$ of degree $d$.
An easy induction shows that 
\[
HS(R,t) = \frac{1}{(1-t)^{d+1}} \mbox{  and  }
  HP(R,k) = {d+k \choose d}
\]
We need one last bit of notation: $R(-i)$ will denote the polynomial 
ring $R$, considered as a graded module over $R$, but with generator 
in degree $i$: 
\[
HS(R(-i),t) = \frac{t^i}{(1-t)^{d+1}} \mbox{  and  } HP(R(-i),k) =
{d+k-i\choose d}.
\] 
\begin{lem}\label{lem:BR}$[$Billera-Rose, \cite{br}$]$
Let $P$ be a $d$--dimensional polyhedral complex. Then
there is a graded exact sequence:
\[
0\longrightarrow \CR \longrightarrow R^{f_d}\oplus
R^{f_{d-1}^0}(-r-1) \stackrel{\phi}{\longrightarrow}
R^{f_{d-1}^0}
\longrightarrow N \longrightarrow
0
\]
\[
 \hbox{where  } \phi = \;\;{\small \left[ \partial_d \Biggm| \begin{array}{*{3}c}
l_{\tau_1}^{r+1} & \  & \  \\
\ & \ddots & \  \\
\ & \ & l_{\tau_m}^{r+1}
\end{array} \right]}
\]
\end{lem}
Write $[\partial_d \mid D]$ for $\phi$. To describe $\partial_d$, note 
that the rows of $\partial_d$ are indexed by $\tau \in P_{d-1}^0$. 
If $\sigma_1, \sigma_2$ denote the $d-$faces adjacent to $\tau$, 
then in the row corresponding to
$\tau$ the smoothness condition means that the only nonzero entries
occur in the columns corresponding to $\sigma_1, \sigma_2$, and are
$\pm(+1,-1)$. When $P$ is simplicial, 
$\partial_d$ is the top boundary map in the (relative) chain complex. 

A main result of \cite{br} is that $N$ is supported on 
primes of codimension at least two. Combining this with the exact 
sequence of Lemma \ref{lem:BR} allows Billera and Rose to 
determine the first two coefficients of the Hilbert polynomial
(though the result is phrased in terms of the Hilbert series).

\section{Main Theorem}
We begin by sketching our strategy. It 
follows from additivity of the Hilbert polynomial on exact sequences
and Lemma \ref{lem:BR} that obtaining
the coefficient of $k^{d-2}$ in the Hilbert polynomial of $\CR$
is equivalent to obtaining the coefficient of $k^{d-2}$ in the 
Hilbert polynomial of $N$. Since
\[
N \simeq (\!\!\!\bigoplus\limits_{\tau \in P_{d-1}^0} R/l_{\tau}^{r+1})/\partial_d,
\]
every element of $N$ is torsion. Using localization, we first
show that the codimension--$2$ associated primes of $N$ must be
{\em linear}, then give a precise description of which codimension--$2$ 
linear primes actually occur. This leads to an 
explicit description of the submodule of $N$ supported in codimension--$2$. 
Elements of this submodule are the only elements of $N$ which contribute
to the $k^{d-2}$ coefficient of the Hilbert polynomial, and the
formula follows.
\subsection{The codimension two associated primes of $N$}
\begin{lem}\label{lem:linear1}
Any codimension--$2$ prime ideal $J$ associated to $N$ contains a
linear form $l_\tau$, for some $\tau \in P^0_{d-1}$.
\end{lem}
\begin{proof}
From the description 
\[
N \simeq \big( \!\!\!\bigoplus\limits_{\tau \in P_{d-1}^0} R \big/ l_{\tau}^{r+1} \big) \big/ \partial_d,
\]
it follows that if no $l_\tau$ is in $J$, then all the $l_\tau$ are
invertible in $R_J$, so that $N_J$ vanishes.
\end{proof}

\begin{lem}\label{lem:main2}
Let $\xi$ be a codimension--$2$ linear space. 
If $\sigma \in P_d$ has at most one facet whose linear span contains $\xi$, 
then every generator of $N$ corresponding to a facet of 
$\sigma$ is mapped to zero in the localization $N_{I(\xi)}$.
\end{lem}
\begin{proof}
In $R_{I(\xi)}$, any $l_\tau$ such that $\xi \not\subseteq V(l_\tau)$ becomes 
invertible. As $N$ is the cokernel of 
\[
\phi = \;\;{\small \left[ \partial_d \Biggm| \begin{array}{*{3}c}
l_{\tau_1}^{r+1} & \  & \  \\
\ & \ddots & \  \\
\ & \ & l_{\tau_m}^{r+1}
\end{array} \right],}
\]
in the right hand submatrix $D$ of $\phi$, all the forms 
(or all save one) $l_\tau^{r+1}$ such that 
$\tau$ is a facet of $\sigma$ become units. As the column 
of the left hand ($\partial_d$) matrix corresponding to $\sigma$ 
has nonzero entries only in rows corresponding to facets of $\sigma$,
this means that every generator corresponding to a facet 
of $\sigma$ has zero image in the localization, and the 
result follows.
\end{proof}
\begin{thm}\label{thm:linear2}
Any codimension--$2$ prime ideal $J$ associated to $N$ is of the form 
$\langle l_{\tau_1},l_{\tau_2} \rangle$ for $\tau_i \in P_{d-1}^0$
such that $V(l_{\tau_1}, l_{\tau_2})$ has codimension--$2$.
\end{thm}
\begin{proof}
If there do not exist two $l_i$ as above, then by Lemma
\ref{lem:linear1}, $V(J)$ is contained in exactly one hyperplane which is the linear span of $\tau \in
P_{d-1}^0$. Thus, in $R_J$, all but one of the $l_\tau$ become units,
and the proof of Lemma \ref{lem:main2} shows that $N_J$ vanishes.
\end{proof}

Theorem \ref{thm:linear2} gives an explicit set of candidates for the
codimension two primes of $N$. As noted earlier, in \cite{br},
Billera and Rose showed that all associated primes of $N$ have 
codimension at least two (this also follows from the arguments 
above), so the theorem identifies 
all candidates for the associated primes of minimal codimension.
In order to determine exactly which codimension two linear primes are
actually associated to $N$, we introduce a certain dual graph, which
depends not just on the combinatorics of $P$, but also on geometry. In
the simplicial case, the codimension two associated primes are 
exactly the vertices of $P$. It turns out that in the polyhedral
case, the geometry is much more subtle.

\subsection{The graph associated to $P$ and a codimension two subspace}
In this section, we analyze the codimension--$2$ associated primes of
$N$ in greater depth. In the simplicial, 
planar case, the Alfeld-Schumaker formula shows that there is a
contribution to the constant term of the Hilbert polynomial from
the local behaviour at interior vertices. As noted in Example
\ref{exm:syz}, a syzygy on the linear forms adjacent to a single
fixed vertex corresponds to a loop around that vertex.

\pagebreak
\begin{exm}\label{exm:polyII}
For the planar polyhedral complex $P$ appearing in Example \ref{exm:first}, 
let $v_i$ be the vertex associated to the face labelled $f_i$. The
dual graph $G(P)$ is:
\begin{figure}[h]
\begin{center}
\epsfig{file=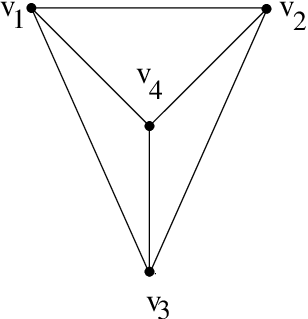,height=1.3in,width=1.5in}
\end{center}
\end{figure}
\end{exm}

\noindent In \cite{r2}, Rose shows that if the 
dual graph $G_P$ of $P$ has a basis
of disjoint cycles, then the projective dimension and Hilbert series
of $\CR$ are determined by the case when $G_P$ has
a single cycle, which Rose analyzed in \cite{r1}. We now
define a refined version of dual graph, which depends on 
$P$ and the choice of a codimension--$2$ linear subspace.
\begin{defn}\label{defn:dualG2}
Let $P$ be a $d$--dimensional polyhedral complex embedded in
$\mathbb{R}^d$, and $\xi$ a
codimension--$2$ linear subspace. $G_\xi(P)$ is a graph
whose vertices correspond to those $\sigma \in P_d$ such 
there exists a $(d-1)$--face of $\sigma$ whose linear span contains
$\xi$. Two vertices of $G_\xi(P)$ are joined iff
the corresponding $d$--faces share a common $(d-1)$--face whose linear
span contains $\xi$.
\end{defn}
\begin{exm}\label{exm:polyIII}
We continue to analyze the complex $P$ of Example \ref{exm:first}. 
For each interior vertex $v$ of $P$, $G_v(P)$ consists of 
a triangle. However, there is another codimension--$2$ 
space to consider: if $\xi$ is the
point at which (the linear hulls of) the three edges connecting
interior vertices to boundary vertices meet, then $G_\xi(P)$ is 

\begin{figure}[h]
\begin{center}
\epsfig{file=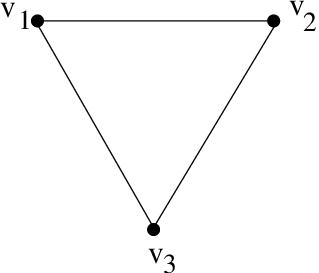,height=1.2in,width=1.4in}
\end{center}
\end{figure}

Let $P'$ be obtained by moving the top vertex of $P$ a bit to the
right. Then $P$ and $P'$ are combinatorially equivalent, but in $P'$
there are no sets of $\ge 3$ concurrent $V(l_\tau), \tau \in
P^0_{d-1}$, except at the interior vertices. In particular, $G_\xi(P')$ is acyclic. 
\end{exm}

\begin{lem}\label{lem:main1}
For any $\sigma \in P_d$, there are at most two facets of $\sigma$ 
whose linear spans contain a given codimension--$2$ linear space $\xi$. 
\end{lem}
\begin{proof}
Suppose the linear spans of three facets $\tau_1,\tau_2,\tau_3$ of $\sigma$
meet in a codimension--$2$ linear space $\xi$. 
For each $V(l_{\tau_i})$, $\sigma$ lies on one 
side of the hyperplane; so $\sigma$ lies between $V(l_{\tau_1})$ 
and $V(l_{\tau_2})$. Since $V(l_{\tau_3})$ contains 
\[
\xi =V(l_{\tau_1}) \cap V(l_{\tau_2}),
\]
this means 
$V(l_{\tau_3})$ would split $\sigma$, a contradiction.
\end{proof}
\begin{cor}\label{cor:cor1}
$G_{\xi}(P)$ is homotopic to a disjoint union of circles and segments.
\end{cor}
\begin{proof}
By Lemma \ref{lem:main1}, the valence of any 
vertex $v \in G_{\xi}(P)$ is at most two.
\end{proof}

\begin{thm}\label{thm:main3}
For a polyhedral complex $P$ and codimension--$2$ linear prime $\xi$,
\[
N_{I(\xi)} \simeq \!\!\!\!\!\! \bigoplus\limits_{\psi \in H_1(G_{\xi}(P))}\!\!\!\!\!\! (R/I_\psi)_{I(\xi)}
\]
where $\psi \in H_1(G_{\xi}(P))$ means $\psi$ is a component of
$G_{\xi}(P)$ homotopic to $S^1$, and 
\[
I_\psi = \langle l_\tau^{r+1} \mid \tau \in P^0_{d-1}
\mbox{ corresponds to an edge of } \psi \rangle.
\]
\end{thm}
\begin{proof}
By Corollary \ref{cor:cor1}, $G_{\xi}(P)$ consists
of a disjoint union of 
cycles and segments. By Lemma \ref{lem:main2}, all generators of
$N$ which lie in a segment are mapped to zero in the localization $N_{I(\xi)}$.
For each $d-$face $\sigma$ corresponding to a vertex in a cycle, note that
there are two $(d-1)-$faces $\tau_1$, $\tau_2$ of $\sigma$ such that 
$l_{\tau_1},l_{\tau_2}$ are not units in $R_{I(\xi)}$; every other
linear form defining a facet of $\sigma$ becomes a unit. Reducing the
column of $\partial_d$ corresponding to $\sigma$ by the columns of $D_{I(\xi)}$
having a unit entry gives a column with nonzero entries only in rows
corresponding to $\tau_1$ and $\tau_2$. Repeating the process shows
that the cycle corresponds to a principal submodule of $N_{I({\xi)}}$, 
with the generator quotiented by the $(r+1)^{st}$ powers of the forms
corresponding to the edges of the cycle.
\end{proof}

\begin{prop}\label{prop:ses}
Let $\mathcal{P}$ be the set of all codimension--$2$ associated primes 
of $N$. Then there is an exact sequence
\[
0 \longrightarrow K \longrightarrow N  \longrightarrow  \!\!\! \!\!\!  \bigoplus\limits_{\stackrel{\psi \in H_1(G_{V(Q)}(P))}{Q \in \mathcal{P}}}  \!\!\! \!\!\! R/I_\psi \longrightarrow C \longrightarrow 0,
\]
where $K$ and $C$ are supported in codimension at least three.
\end{prop}
\begin{proof}
The reasoning in the proof of Theorem~\ref{thm:main3} shows that if 
$\xi = V(Q)$ with $Q \in \mathcal{P}$, then 
\[
\bigoplus\limits_{\psi \in H_1(G_{V(Q)}(P))} \!\!\!\!\!\! R/I_\psi
\]
is exactly
the cokernel of the submatrix of $[\partial_d \mid D ]$ obtained by
deleting those rows indexed by $\tau \in P_{d-1}^0$ such that 
$\xi \not \in \mbox{conv}(\tau)$. An application of the snake lemma
then shows that 
\[
N \longrightarrow \bigoplus\limits_{\psi \in H_1(G_{\xi}(P))} \!\!\!\!\! R/I_\psi  \longrightarrow 0.
\]
is exact. Taking the sum of such maps over all $Q \in \mathcal{P}$ yields
the exact sequence of the proposition. Theorem~\ref{thm:main3} shows
that upon localizing this sequence at any prime $Q \in \mathcal{P}$, 
the localizations $C_{Q}$ and $K_{Q}$ vanish, hence $K$ and $C$ are supported in 
codimension at least three.
\end{proof}

\begin{cor}\label{cor:HP}
Let $M$ be a graded module, and let $a_{d-2}(M)$ denote the
coefficient of $k^{d-2}$ in $HP(M,k)$. Let $\mathcal{P}$ be 
the set of codimension--$2$ primes associated to $N$. 
Then $a_{d-2}(\CR) =$
\[
a_{d-2}(R^{f_d-f_{d-1}^0})+a_{d-2}(R(-r-1)^{f_{d-1}^0})+\sum\limits_{Q
  \in \mathcal{P}} \sum\limits_{\psi_j \in H_1(G_{V(Q)}(P))}
a_{d-2}(R/I_{\psi_j}).
\]
\end{cor}
\begin{proof}
By Lemma~\ref{lem:BR} and the additivity of Hilbert polynomials on 
exact sequences, 
\[
a_{d-2}(\CR) =
a_{d-2}(R^{f_d-f_{d-1}^0})+a_{d-2}(R(-r-1)^{f_{d-1}^0})+a_{d-2}(N).
\]
By Proposition~\ref{prop:ses},
\[
a_{d-2}(N) = \sum\limits_{Q \in \mathcal{P}}\sum\limits_{\psi_j \in H_1(G_{V(Q)}(P))} a_{d-2}(R/I_{\psi_j}),
\]
and the result follows. 
\end{proof}
The value of $a_{d-2}(R(-m))$ is a Stirling number of the first 
kind, and can be written out explicitly. In order for Corollary~\ref{cor:HP} to be useful, we need to know $a_{d-2}(R/I_\psi)$,
which is provided by the following lemma: 

\begin{lem}\label{lem:syzlines}
Let $I_\psi = \langle l_1^{r+1},\ldots,l_n^{r+1} \rangle \subseteq
\mathbb{K}[x_0,\ldots,x_d]$ be a codimension--$2$ ideal, minimally generated
by the $n$ given elements. Define
\[
\begin{array}{ccc}\alpha(\psi) & =&\lfloor\frac{r+1}{n-1}\rfloor,\\
 s_1(\psi)&=&(n\!-\!1)\alpha(\psi)\!+\!n\!-\!r\!-\!2,\\ 
s_2(\psi)&=&r\!+\!1\!-\!(n\!-\!1)\alpha(\psi).
\end{array}
\]
Then the minimal free resolution of $R/I_\psi$ is:
\[
0  \longrightarrow
\begin{array}{c}
R(-r\!-\!1\!-\!\alpha(\psi))^{s_1(\psi)} \\
\oplus\\
R(-r\!-\!2\!-\!\alpha(\psi))^{s_2(\psi)} 
\end{array}
\!\!\! \longrightarrow
R(-r\!-\!1)^n
 \longrightarrow
R \longrightarrow
R/I_\psi \longrightarrow  0.
\]
\end{lem}
\begin{proof}
See Theorem 3.1 of \cite{ss1}; the key step involves showing that
a certain matrix has full rank, which was established by Schumaker in
\cite{schu}.
\end{proof}
\noindent It follows from Lemma \ref{lem:syzlines} that 
the Hilbert polynomial of $R/I_\psi$ is given by:

\[
\binom{k\!+\!d}{d}-n\binom{k\!+\!d\!-\!r\!-\!1}{d}  +
s_1(\psi)\binom{k\!+\!d\!-\!r\!-\!1\!-\!\alpha(\psi)}{d} +
s_2(\psi)\binom{k\!+\!d\!-\!r\!-\!2\!-\!\alpha(\psi)}{d}.
\]
\vskip .05in
\begin{cor}\label{cor:planarHP}
If $P$ is a hereditary planar polyhedral complex, then 
\[
HP(\CR,k) = \frac{f_2}{2}k^2 + \frac{3f_2-2(r+1)f_1^0}{2}k + f_2+ \Big({r
  \choose 2}-1\Big)f_1^0 + \!\!\! \!\!\! \sum\limits_{\psi_j \in H_1(G_{\xi_i}(P))}\!\!\! c_j,
\]
where
\[
c_j = 1-n(\psi_j){r \choose 2} + s_1(\psi_j){r + \alpha(\psi_j)
  \choose 2} + s_2(\psi_j){r + \alpha(\psi_j)+1 \choose 2} 
\]
\[
= {r+2 \choose 2}+\frac{\alpha(\psi_j)}{2}\Big(2r+3+\alpha(\psi_j)-n(1+\alpha(\psi_j))\Big).
\]
\end{cor}
\section{Examples and connection to simplicial case}
We close with some examples, and a discussion of the relation
to the simplicial case. We begin by applying Corollary 
\ref{cor:planarHP} to Example
\ref{exm:first}. As we saw in Example \ref{exm:polyIII}, there
are four $\xi$ at which $H_1(G_\xi(P)) \ne 0$, and each 
$I_\psi$ has three generators. Hence 
the $c_j$ are all the same, and equal to 
\[
{r+2 \choose
  2}+\frac{\alpha(\psi_j)}{2}\Big(2r+3+\alpha(\psi_j)-3(1+\alpha(\psi_j))\Big).
\]
which simplifies to 
\[
 {r+2 \choose 2}+
 \lfloor\frac{r+1}{2}\rfloor\Big(r-\lfloor\frac{r+1}{2}\rfloor \Big).
\]
Applying Corollary \ref{cor:planarHP} yields:
\begin{center}
\begin{supertabular}{|c|c|c|c|c|}
\hline $r$ & $\dim_{\mathbb{R}} C^r_k(P)$ & $ \frac{f_2}{2}k^2 + \frac{3f_2-2(r+1)f_1^0}{2}k$ & $f_2+ \Big({r
  \choose 2}-1\Big)f_1^0$ & $4({r+2 \choose 2}+ \alpha(r-\alpha))$\\
\hline $0$ & $2k^2+2$       & $2k^2$     & $-2$ & $4$\\
\hline $1$ & $2k^2-6k+10$   & $2k^2-6k$  & $-2$ & $12$\\
\hline $2$ & $2k^2-12k+32$  & $2k^2-12k$ & $4 $ & $28$\\
\hline $3$ & $2k^2-18k+64$  & $2k^2-18k$ & $16$ & $48$\\
\hline $4$ & $2k^2-24k+110$ & $2k^2-24k$ & $34$ & $76$\\
\hline
\end{supertabular}
\end{center}
\vskip .4in
As in Example \ref{exm:polyIII}, consider the configuration 
$P'$ obtained by perturbing a vertex in Example \ref{exm:first} 
so the three edges defining $\xi$ no longer meet. Then there
are only three nontrivial $c_j$, and we have

\begin{center}
\begin{supertabular}{|c|c|c|c|c|}
\hline $r$ & $\dim_{\mathbb{R}} C^r_k(P')$ & $ \frac{f_2}{2}k^2 + \frac{3f_2-2(r+1)f_1^0}{2}k$ & $f_2+ \Big({r
  \choose 2}-1\Big)f_1^0$ & $3({r+2 \choose 2}+ \alpha(r-\alpha))$\\
\hline $0$ & $2k^2+1$       & $2k^2$     & $-2$ & $3$\\
\hline $1$ & $2k^2-6k+7$   & $2k^2-6k$  & $-2$ & $9$\\
\hline $2$ & $2k^2-12k+25$  & $2k^2-12k$ & $4 $ & $21$\\
\hline $3$ & $2k^2-18k+52$  & $2k^2-18k$ & $16$ & $36$\\
\hline $4$ & $2k^2-24k+91$ & $2k^2-24k$ & $34$ & $57$\\
\hline
\end{supertabular}
\end{center}
\vskip .1in

\begin{exm}\label{exm:2squares}
In this example, we show that $G_{\xi}(P)$ can have several disjoint
cycles. Let $P$ be as below:

\begin{figure}[h]
\begin{center}
\epsfig{file=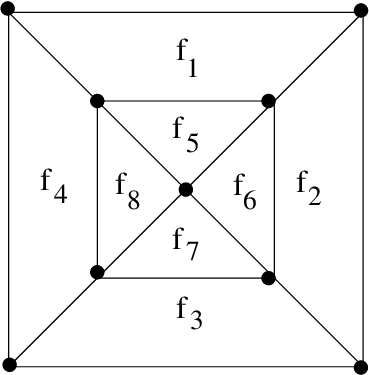,height=1.5in,width=1.5in}
\end{center}
\end{figure}

\noindent If $\xi$ corresponds to the central vertex, then $G_{\xi}(P)$ 
is:

\begin{figure}[h]
\begin{center}
\epsfig{file=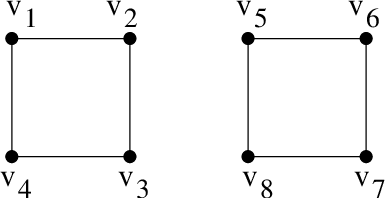,height=1.2in,width=2.0in}
\end{center}
\end{figure}
\noindent For these two cycles, $I_\psi$ has only 
two generators, and the $c_j$ value for such an ideal is always
$(r+1)^2$. There are four cycles for which $I_\psi$ has three generators, so 
by Corollary \ref{cor:planarHP} we have (this table omits column for 
$\frac{f_2}{2}k^2 + \frac{3f_2-2(r+1)f_1^0}{2}k$):
\vskip .1in
\begin{center}
\begin{supertabular}{|c|c|c|c|c|}
\hline $r$ & $\dim_{\mathbb{R}} C^r_k(P)$ &  $f_2+ \Big({r
  \choose 2}-1\Big)f_1^0$ & $4({r+2 \choose 2}+ \alpha(r-\alpha))$ & $2(r+1)^2$\\
\hline $0$ & $4k^2+2$       &  $-4$ & $4$  &$2$  \\
\hline $1$ & $4k^2-12k+16$   &  $-4$ & $12$ &$8$  \\
\hline $2$ & $4k^2-24k+54$  &  $8 $ & $28$ &$18$  \\
\hline $3$ & $4k^2-36k+112$  &  $32$ & $48$ &$32$  \\
\hline $4$ & $4k^2-48k+194$ &  $68$ & $76$ &$50$  \\
\hline
\end{supertabular}
\end{center}
\vskip .1in
\end{exm}

\begin{exm}\label{exm:bees}
In our final example, we look at a honeycomb configuration, that
is, the polyhedral complex consisting of seven hexagons, pictured
below:
\vskip .1in

\begin{figure}[h]
\begin{center}
\epsfig{file=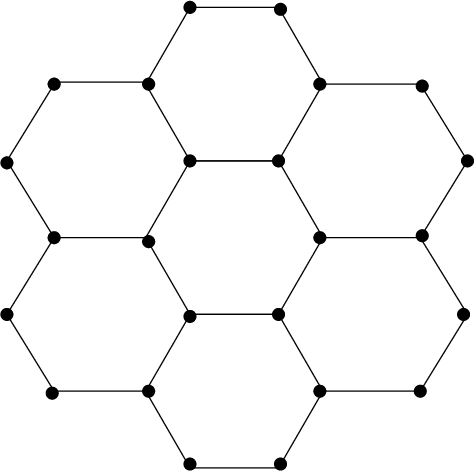,height=1.5in,width=1.7in}
\end{center}
\end{figure}

For each interior vertex $v$, $G_v(P)$ will be a triangle and a line
segment; while if $\xi$ is the point at the center of the diagram,
 $G_\xi(P)$ is a hexagon. However, each of the $I_\psi$ is generated
by three elements, and so 
Corollary \ref{cor:planarHP} yields:
\begin{center}
\begin{supertabular}{|c|c|c|c|c|}
\hline $r$ & $\dim_{\mathbb{R}} C^r_k(P)$ & $ \frac{f_2}{2}k^2 + \frac{3f_2-2(r+1)f_1^0}{2}k$ & $f_2+ \Big({r
  \choose 2}-1\Big)f_1^0$ & $7({r+2 \choose 2}+ \alpha(r-\alpha))$\\
\hline $0$ & $\frac{7}{2}k^2-\frac{3}{2}k    + 2$   & $\frac{7}{2}k^2-\frac{3}{2}k$ & $-5$ & $7$\\
\hline $1$ & $\frac{7}{2}k^2-\frac{27}{2}k   +16$   & $\frac{7}{2}k^2-\frac{27}{2}k$  & $-5$ & $21$\\
\hline $2$ & $\frac{7}{2}k^2-\frac{51}{2}k   +56$  & $\frac{7}{2}k^2-\frac{51}{2}k$ & $7 $ & $49$\\
\hline $3$ & $\frac{7}{2}k^2-\frac{75}{2}k   +115$  & $\frac{7}{2}k^2-\frac{75}{2}k$ & $31$ & $84$\\
\hline $4$ & $\frac{7}{2}k^2-\frac{99}{2}k   +200$  & $\frac{7}{2}k^2-\frac{99}{2}k$  & $67$ & $133$\\
\hline
\end{supertabular}
\end{center}
\vskip .1in
\end{exm}
\subsection{Connection to the simplicial case}
In the case where $P$ is actually a simplicial complex, the
first point to notice is that when two facets $\tau_1,\tau_2$ of 
$\sigma \in P_d$ meet, then they actually meet in a face of $\sigma$,
necessarily of codimension--$2$. Thus, in the simplicial case, 
the only $\xi$ such that $H_1(G_\xi(P)) \ne 0$ are $\xi \in
P^0_{d-2}$. For a planar $P$ which is simplicial, there are $f^0_0$
such faces. The formula for the value $c_j$ which appears in 
Corollary \ref{cor:planarHP} can be rewritten as:
\[
{r+2 \choose 2}+\sigma_i,
\]
where $\sigma_i$ is as in the Alfeld-Schumaker formula. So 
in the planar, simplicial case, the formula of Corollary 
\ref{cor:planarHP} is the polynomial appearing in Theorem 1.1.
\vskip .05in
\noindent {\bf Concluding Remarks and Questions}:  
\begin{enumerate}
\item To obtain the next coefficient of the Hilbert polynomial will
be difficult; indeed, even in the simplicial case no general formula
is known.
\item As noted in \cite{ss}, the value for which the dimension of 
$C^r_k(P)$ becomes polynomial is the Castelnuovo-Mumford 
regularity of the sheaf associated to $\CR$; it would be interesting to
determine the regularity of $\CR$.
\item It is possible to generalize these results to the case of
mixed smoothness, which was considered for planar simplicial complexes
in \cite{gs}; we leave this to the interested reader.
\item In the simplicial, planar case, \cite{ds} shows that $\CR \mbox{
  free} \rightarrow C^{r-1}(\hat P) \mbox{ free}$. Is this true for
  $P$ polyhedral?
\end{enumerate}
\vskip .05in
\noindent {\bf Acknowledgments}:  Macaulay2 \cite{danmike} computations were
essential to our work. 
\renewcommand{\baselinestretch}{1.0}
\small\normalsize 

\bibliographystyle{amsalpha}

\end{document}